\documentclass{article}
\usepackage{amsmath,amssymb,amsthm}
\usepackage{pstricks,pst-node}
\usepackage{graphicx}

\newtheorem{thm}{Theorem}[section]

\newtheorem{defn}[thm]{Definition}
\newtheorem{lem}[thm]{Lemma}

\newtheorem{conj}[thm]{Conjecture}

\renewcommand{\And}{\operatorname{And}}
\newcommand{\Compn}{\overline{C_n}}
\newcommand{\F}{\mathcal{F}}
\newcommand{\td}{\operatorname{td}}

\title{On 1-uniqueness and dense critical graphs for tree-depth}

\author{Michael D. Barrus\thanks{Department of Mathematics, University of Rhode Island, Kingston, Rhode Island 02881, United States; email: \texttt{barrus@uri.edu}} ~and
John Sinkovic\thanks{Department of Combinatorics and Optimization, University of Waterloo, Waterloo, Ontario, N2L 2G1, Canada; email: \texttt{johnsinkovic@gmail.com}}}

\begin{document}

\maketitle

\begin{abstract}
The tree-depth of $G$ is the smallest value of $k$ for which a labeling of the vertices of $G$ with elements from $\{1,\dots,k\}$ exists such that any path joining two vertices with the same label contains a vertex having a higher label.  The graph $G$ is $k$-critical if it has tree-depth $k$ and every proper minor of $G$ has smaller tree-depth.

Motivated by a conjecture on the maximum degree of $k$-critical graphs, we consider the property of 1-uniqueness, wherein any vertex of a critical graph can be the unique vertex receiving label 1 in an optimal labeling. Contrary to an earlier conjecture, we construct examples of critical graphs that are not 1-unique and show that 1-unique graphs can have arbitrarily many more edges than certain critical spanning subgraphs. We also show that $(n-1)$-critical graphs are 1-unique and use 1-uniqueness to show that the Andr\'{a}sfai graphs are critical with respect to tree-depth.

\end{abstract}

\section{Introduction}
The tree-depth of a graph $G$ is the smallest size of a set $\{1,\dots,n\}$ of labels with which the vertices of $G$ may be labeled so that any path between two vertices with the same label contains a vertex receiving a higher label. Equivalently, the tree-depth is the minimum number of steps needed to delete all vertices from $G$ if at each step at most one vertex may be deleted from each current component. Tree-depth is denoted $\td(G)$ and has also been called the vertex ranking number or ordered chromatic number of $G$. For a sampling of results on tree-depth, and bibliographic references to other sources, see~\cite{BarNoyEtAl12,path,BodlaenderEtAl98,ChangEtAl10,IyerEtAl88,NesetrilOssonadeMendez12,NesetrilOssonadeMendez06,NesetrilOssonadeMendez08}.

One fundamental property of tree-depth is its monotonicity under the graph minor relationship; as noted in~\cite{NesetrilOssonadeMendez12}, if $G$ is a minor of $H$, then $\td(G) \leq \td(H)$. If $M$ is a graph with tree-depth $k$ such that every proper minor of $M$ has tree-depth less than $k$, we say that $M$ is \emph{minor-critical}, or simply \emph{critical} or \emph{$k$-critical}. (For clarity, we note here that some authors use ``critical'' when discussing the context of subgraphs, rather than our present context of minors.)

Because of the monotonicity of tree-depth under the minor relationship, it follows that every graph with tree-depth $k$ has a $k$-critical minor, and the graphs with tree-depth at most $k$ are characterized in terms of a list of forbidden minors; the minimal such list consists of all $(k+1)$-critical graphs.

In~\cite{DGT}, Dvo{\v{r}}{\'a}k, Giannopoulou, and Thilikos initiated the study of critical graphs having small tree-depth. They determined all $k$-critical graphs for $k \leq 4$ and exhibited a construction of critical graphs from smaller ones; this construction is sufficient to construct all critical trees of any tree-depth.

In examining the critical graphs with small tree-depth, a number of apparent patterns suggest themselves. We mention two conjectured relationships.
\begin{conj} \label{conj: order, max degree} Let $G$ be a $k$-critical graph.
\begin{enumerate}
\item[\textup{(a)}] \textup{(\cite{DGT})} $G$ has at most $2^{k-1}$ vertices.
\item[\textup{(b)}] \textup{(\cite{BarrusSinkovic15})} $G$ has maximum degree at most $k-1$.
\end{enumerate}
\end{conj}
\noindent Both conjectures presently remain open in their full generality.

In~\cite{BarrusSinkovic15}, the authors observed that item (b) above is true for a special class of graphs, the 1-unique graphs. A graph $G$ is \emph{1-unique} if for every vertex $v$ in $G$, there exists a labeling of the vertices of $G$ with labels from $\{1,\dots,\td(G)\}$ having the defining requirement of tree-depth and also having the property that $v$ is the unique vertex of $G$ receiving the label 1. A 1-unique $k$-critical graph must have maximum degree $k-1$, since the neighbors of any vertex labelled 1 must receive distinct labels.

In~\cite{BarrusSinkovic15} the authors established some similarities between tree-depth criticality and 1-uniqueness. In particular, they showed that though 1-unique graphs need not be critical, if $G$ is any 1-unique graph, then $G$ has a subset of edges that can be removed so as to leave a spanning subgraph of $G$ that is $\td(G)$-critical. Furthermore, every critical graph with tree-depth at most 4 is 1-unique, and the authors noted that every critical graph constructed through a certain generalization of the algorithm in~\cite{DGT} is also 1-unique. These facts led the authors to the following.

\begin{conj} \label{conj: false}
Every critical graph is 1-unique.
\end{conj}

As it happens, Conjecture~\ref{conj: false} is false infinitely often, as we will shortly show. 
In Section 2 we discuss a computer search that found a number of counterexamples, one of which we generalize to an infinite family of non-1-unique critical graphs. In Section 3 we show that the edge subset that can be deleted from a 1-unique graph to leave a critical spanning subgraph may be arbitrarily large. These results have the effect of somewhat separating the properties of criticality and 1-uniqueness, which appeared in~\cite{BarrusSinkovic15} to be closely related.

However, we will also show that the property of 1-uniqueness can be used to good effect in questions on criticality. In Section 4 we use 1-uniqueness to efficiently show that the Andr\'{a}sfai graphs are (1-unique) critical graphs. 

The examples and counterexamples considered in this paper follow a theme of dense graphs (where we take `dense' here to mean that the number of edges in $n$-vertex members of the family is at least a constant fraction of $\binom{n}{2}$). This is true for the family of counterexamples to Conjecture~\ref{conj: false} presented in Section 2, as well as for the families of graphs in Section 3 that illustrate differences between the number of edges in 1-unique graphs and their critical subgraphs. These examples stand in contrast to the 1-unique critical trees and other often-sparse critical graphs shown and constructed in~\cite{DGT} and~\cite{BarrusSinkovic15}, which motivated Conjecture~\ref{conj: false}. On the other hand, the relationship between tree-depth criticality and 1-uniqueness cannot be simply explained by sparseness, as in Section 2 we show that the $(n-1)$-critical graphs are both dense and 1-unique, and the Andr\'{a}sfai graphs in Section 4 are likewise examples of dense, 1-unique, critical graphs.

We define a \emph{labeling} of a graph $G$ to be an assignment of the vertices of $G$ with the positive integers. If every path between any two vertices with the same label contains a vertex having a higher label, we call the labeling \emph{feasible}; thus $\td(G)$ is the smallest number of labels necessary for a feasible labeling. We use $V(G)$ and $N_G(v)$ to denote the vertex set of $G$ and the (open) neighborhood in $G$ of vertex $v$,  respectively. The complete graph and cycle with $n$ vertices are referred to respectively as $K_n$ and $C_n$. The Cartesisan product of graphs $G$ and $H$ is denoted by $G \Box H$, the disjoint union of $G$ and $H$ is denoted by $G+H$, and the disjoint union of $m$ copies of $G$ is denoted by $mG$.

\section{Critical graphs with high tree-depth}
In this section we consider graphs whose tree-depth differs from the number of vertices by at most a fixed constant. The extremal example is the complete graph $K_n$, which is the unique graph on $n$ vertices where these parameters are equal. We show that graphs with similar relatively high tree-depths form hereditary classes of graphs. These results will help at the end of this section, where we prove a special case of Conjecture~\ref{conj: false}, that critical graphs with tree-depth almost as high as that of $K_n$ are 1-unique. The results will also be useful in the next section, where we separate the notions of tree-depth criticality and 1-uniqueness.

A graph class is hereditary, i.e., closed under taking induced subgraphs, if and only if it can be characterized by a collection of forbidden induced subgraphs. We show that this is true of graphs with relatively high tree-depths compared to their respective numbers of vertices.

\begin{thm} \label{thm: high td hereditary}
Given a nonnegative integer $k$, let $\F_k$ denote the set of minimal elements, under the induced subgraph ordering, of all graphs $H$ for which $n(H)- \td(H) = k+1$. For all graphs $G$, the graph $G$ satisfies $\td(G) \geq n(G) - k$ if and only if $G$ is $\F_k$-free.
\end{thm}

To illustrate this theorem, we identify a few special cases, some of which will be useful later in the paper.

\begin{lem} \label{lem: high td forb sub lists}
 Let $G$ be a graph with $n$ vertices.
\begin{enumerate}
\item[\textup{(a)}] The graph $G$ satisfies $\td(G) = n$ if and only if $G$ is $\{2K_1\}$-free.
\item[\textup{(b)}] The graph $G$ satisfies $\td(G) \geq n-1$ if and only if $G$ is $\{3K_1,2K_2\}$-free.
\item[\textup{(c)}] The graph $G$ satisfies $\td(G) \geq n-2$ if and only if $G$ is $\{4K_1,2K_2+K_1,P_3+K_2,2K_3\}$-free.
\end{enumerate}
\end{lem}

We remark in passing that Lemma~\ref{lem: high td forb sub lists} suggests that the graphs with nearly equal tree-depths and orders are necessarily dense graphs.

To prove Theorem~\ref{thm: high td hereditary} and Lemma~\ref{lem: high td forb sub lists}, we first develop some terminology and a few preliminary results. For any graph $G$, define the \emph{surplus} $s(G)$ by $s(G) = n(G) - \td(G)$. Clearly $\td(G) \geq n-k$ if and only if $s(G) \leq k$, and the graphs in $\F_k$ are the minimal graphs under induced subgraph inclusion for which the surplus is $k+1$.

We now show that $s(G)$ is well-behaved under taking induced sugraphs.

\begin{proof}[Proof of Theorem~\ref{thm: high td hereditary}]
We prove the equivalent statement that $G$ has no induced subgraph with surplus $k+1$ if and only if $s(G)\leq k$.

Suppose first that $s(G) > k$. If $s(G)=k+1$ then $G$ clearly has at least one induced subgraph with surplus $k+1$. If $s(G)>k+1$, then let $v_1,\dots,v_n$ denote the vertices of $G$ under some arbitrary ordering. If we define $G_0=G$ and $G_i = G - \{v_1,\dots,v_i\}$ for each $i \in \{1,\dots,n-1\}$, then for $i\geq 1$ each $G_i$ is obtained by deleting a vertex from $G_{i-1}$. Since deleting a vertex from a graph either leaves the tree-depth unchanged or lowers it by one, it follows that $s(G_i) = s(G_{i-1})-1$ or $s(G_i)=s(G_{i-1})$. Since $s(G) > k+1$ and $G_{n-1}\cong K_1$ and hence $s(G_{n-1})=0$, it follows that $s(G_j)=k+1$ for some $j \in \{1,\dots,n-2\}$. Hence $G$ has an induced subgraph with surplus $k+1$.

Suppose instead that $G$ has an induced subgraph $H$ with surplus $k+1$. As described in the previous paragraph, deleting a single vertex from a graph either maintains the present value of the surplus or decreases it by 1, so if we imagine deleting vertices from $V(G)\setminus V(H)$ to arrive at the induced subgraph $H$, we have $s(G) \geq s(H)=k+1$, so $s(G)>k$.
\end{proof}

Before proving Lemma~\ref{lem: high td forb sub lists}, we develop a useful type of labeling. Call a feasible labeling of a graph \emph{reduced} if no label appearing on more than one vertex has a higher value than a label appearing on at most one vertex. In other words, in a reduced labeling, the repeated labels are the lowest.

\begin{lem}\label{lem: reduced labelings} \mbox{}
\begin{enumerate}
\item[(1)] Every graph $G$ has an optimal reduced labeling.
\item[(2)] If $\gamma$ is an optimal reduced labeling of $G$, and $H$ is the induced subgraph of $G$ consisting of all vertices sharing their label with some other vertex of $G$, then the restriction of $\gamma$ to $H$ is an optimal labeling of $H$; hence $s(H) = s(G)$ and $\td(H) \leq s(H)$.
\end{enumerate}
\end{lem}
\begin{proof}
Let $\gamma$ be a feasible labeling of $G$ using $\td(G)$ labels, and let $\ell_1,\dots,\ell_k$ denote the labels appearing on more than one vertex in $\gamma$, ordered so that $\ell_1<\dots<\ell_k$. Construct an alternate labeling $\delta$ of the vertices of $G$ as follows: for $i \in \{1,\dots,k\}$, assign label $i$ to each of the vertices having label $\ell_i$ in $\gamma$; then label the remaining unlabeled vertices arbitrarily but injectively with labels from $\{k+1,\dots,\td(G)\}$.

We claim that $\delta$ is an optimal feasible labeling of $G$. First, note that $\delta$ uses the same number of labels as the optimal labeling $\gamma$. Next note that any path in $G$ between vertices with the same label from $\delta$ corresponds to a path joining vertices with the same label from $\gamma$. In $\gamma$ some vertex $v$ on this path received a higher label. If $v$ did not share its $\gamma$-label with another vertex in $G$, then in the labeling $\delta$ the vertex $v$ receives a label higher than $k$, while the endpoints of the path receive a label less than or equal to $k$. If $v$ does share its $\gamma$-label with another vertex in $G$, then by construction since $v$'s label in $\gamma$ was higher than that of the path endpoints', this same relationship holds between the vertex labels in $\delta$. We conclude that $\delta$ is a feasible labeling of $G$; hence $G$ has an optimal labeling that is reduced, establishing (1).

Supposing now that $\gamma$ is a reduced optimal labeling of $G$, let $H$ be the graph formed by deleting from $G$ all vertices not sharing their label with some other vertex of $G$. Let $\gamma'$ denote the restriction of $\gamma$ to the graph $H$. We claim that $\gamma'$ is an optimal labeling of $H$.

First, observe that if $H$ has an optimal labeling $\beta$ using fewer colors than $\gamma'$, then we may modify the labeling $\gamma$ of $G$ by replacing the labels on vertices of $H$ with the labels from $\beta$. Since $\gamma$ was a reduced labeling, the resulting modification is still a feasible labeling and uses fewer than $\td(G)$ labels, a contradiction.

On the other hand, it is simple to verify that $\gamma'$ is a feasible labeling of $H$, so in fact $\td(H) = \td(G)-(n(G)-n(H))$, and $s(H)=s(G)$. Finally, consider a subset $W$ of $V(G)$ containing, for each label used by $\gamma$, a single vertex of $G$ receiving that label. The labels appearing on vertices in $H$ are precisely the labels $\gamma$ assigns to vertices in $V(G)-W$, so $\td(H) \leq s(G) = s(H)$, establishing (2).
\end{proof}

Given a graph $G$ and a feasible labeling $\gamma$ of $G$ as in Lemma~\ref{lem: reduced labelings}, let the graph $H$ described in the lemma be the \emph{irreducible core of $G$ under $\gamma$}. Further call $G$ \emph{irreducible under $\gamma$} if every label of $\gamma$ appears on at least two vertices of $G$. Note that any graph that is irreducible under $\gamma$ is disconnected, since there is at most one vertex per component having the maximum label in any feasible labeling.

\begin{proof}[Proof of Lemma~\ref{lem: high td forb sub lists}]
Using the language of Theorem~\ref{thm: high td hereditary}, we need simply show that $\F_0=\{2K_1\}$, $\F_1=\{3K_1,2K_2\}$, and $\F_2=\{4K_1,2K_2+K_1,P_3+K_2,2K_3\}$.

Note first that $s(2K_1) = 1$, that $s(3K_1) = s(2K_2) = 2$, and that for each $G \in \{4K_1,2K_2+K_1,P_3+K_2,2K_3\}$ we have $s(G) = 3$. It is straightforward to check that if $H$ is any proper induced subgraph of any one of the graphs listed in the statement of the lemma, then the surplus of $H$ is less than the surplus of the graph. Hence $\{2K_1\}\subseteq \F_0$, $\{3K_1,2K_2\}\subseteq \F_1$, and $\{4K_1,2K_2+K_1,P_3+K_2,2K_3\} \subseteq \F_2$.

Suppose now that $G$ is a graph for which $s(G) \in \{1,2,3\}$, and let $H$ be the irreducible core of $G$ under some optimal labeling $\gamma$. By Lemma~\ref{lem: reduced labelings} the restriction of $\gamma$ to $H$ is an optimal labeling of $H$, and $\td(H)\leq s(H)=s(G)$. Observe that $H$ contains at least $s(G)+1$ vertices, and every label in $H$ appears on at least two vertices.

If $s(G)=1$, then we claim that $H$ contains $2K_1$ as an induced subgraph. Indeed, $H$ contains at least two vertices receiving the same label under $\gamma$. Since $\gamma$ is a proper coloring of the vertices of $H$, these two vertices are nonadjacent, and thus $H$ and $G$ induce $2K_1$. We conclude that $\F_0=\{2K_1\}$.

If $s(G)=2$, then we claim that $H$ contains an induced subgraph from $\{3K_1,2K_2\}$. If $\td(H)=1$, then since $H$ has at least three vertices, $H$ contains $3K_1$ as an induced subgraph. If instead $\td(H)=2$, then since every label appears on at least two vertices of $H$, the graph $H$ contains four vertices $a,b,c,d$ such that $a$ and $b$ both receive label 1 and $c$ and $d$ receive label 2. We also know that $H$ is $\{P_4,C_4,K_3\}$-free, since $\td(H)<3$. Considering all 4-vertex graphs directly, we see that $H[\{a,b,c,d\}]$ either induces $3K_1$ or is isomorphic to $2K_2$. Hence $\F_1=\{3K_1,2K_2\}$.

If $s(G)=3$, we claim that $H$ induces an element of $\{4K_1,2K_2+K_1,P_3+K_2,2K_3\}$. Observe that if $\td(H)=1$, then as $H$ has at least four vertices, $4K_1$ is an induced subgraph. If $\td(H)=2$, then $H$ is a forest of stars; we can also say that $H$ has five vertices, and it has at least two components that have an edge. The only two such graphs are $2K_2+K_1$ and $P_3+K_2$, so $H$ is one of these. If $\td(H)=3$, then $H$ has six vertices. Since $H$ is the irreducible core of $G$ under $\gamma$, it follows that $H$ consists of exactly two components, each having tree-depth 3; the only such graph on six vertices is $2K_3$, so $\F_2=\{4K_1,2K_2+K_1,P_3+K_2,2K_3\}$, and the proof is complete.
\end{proof}

We turn our attention now to critical graphs. As we do so, it is interesting to observe that Theorem~\ref{thm: high td hereditary} provides a contrasting result on how substructures affect tree-depth. Critical graphs are the minors that force a graph to have \emph{higher} tree-depth, while as in the proof of Theorem~\ref{thm: high td hereditary} the graphs in $\mathcal{F}_k$ are precisely the subgraphs that allow for labelings with a \emph{smaller} number of labels.

We recall a definition and result from~\cite{BarrusSinkovic15}.  Given a vertex $v$ in a graph $G$, a \emph{star-clique transform at $v$} removes $v$ from $G$ and adds edges between the vertices in $N_G(v)$ so as to make them a clique.  

\begin{thm}\label{thm:starclique}(\cite{BarrusSinkovic15}) Let $v$ be a vertex of a graph $G$, and let $H$ be the graph obtained through the star-clique transform at $v$ of $G$.  Vertex $v$ is $1$-unique in $G$ if and only if $\td(H)<\td(G)$.
\end{thm}

\begin{lem} \label{lem: star-clique on 3K1,2K2 free}
If $G$ is a $\{3K_1,2K_2\}$-free graph and $G'$ is obtained via a star-clique transform of $G$, then $G'$ is also $\{3K_1,2K_2\}$-free.
\end{lem}
\begin{proof}
We prove the contrapositive.  Let $G'$ be a graph that is not $\{3K_1,2K_2\}$-free and that is obtained from $G$ by a star-clique transform at vertex $v$.  By the definition of a star-clique transform, $V(G')=V(G-v)$ and $E(G-v)\subset E(G')$.  Thus, if $3K_1$ is induced in $G'$, then it is induced in $G-v$.  Since any induced subgraph of $G-v$ is also induced in $G$, $3K_1$ is induced in $G$ as well.  

Similarly, if $2K_2$ is induced in $G'$ and $G-v$, then it is induced in $G$. Thus we consider the case where $2K_2$ is induced in $G'$ but not in $G-v$. Let $H$ be such a subgraph of $G'$. At least one edge of $H$ is in $E(G')\setminus E(G-v)$.  Since $E(G')$ and $E(G-v)$ differ only in the edges joining vertices in $N_G(v)$, $|N_G(v)\bigcap V(H)|\geq 2$. Since $N_G(v)$ induces a clique in $G'$, $|N_G(v)\bigcap V(H)|\leq 2$.  Thus $H$ has two adjacent vertices in $N_G(v)$ and two adjacent vertices not in $N_G(v)$. The latter two vertices, together with a third vertex from $H$ and the vertex $v$, induce $2K_2$ in $G$.
\end{proof}

\begin{thm} \label{thm: n-1 critical is 1-unique}
If $G$ is a critical graph with $n$ vertices and $\td(G) \geq n-1$, then $G$ is 1-unique.
\end{thm}
\begin{proof}
The only $n$-vertex graph with tree-depth $n$ is $K_n$, which is 1-unique, so suppose that $G$ is $(n-1)$-critical. If $G$ is not 1-unique, then by Theorem~\ref{thm:starclique} there exists a vertex $v$ in $G$ such that performing a star-clique transform on $G$ at $v$ results in a graph $G'$ such that $\td(G') \geq \td(G)$. Then $n-1=\td(G) \leq \td(G') \leq n(G')=n-1$, so $\td(G')=n(G')$, and in fact $G'$ is a complete graph. It follows that in $G$ each vertex not in $N_G(v)$ is adjacent to every vertex of $N_G(v)$. Now let $e$ be an edge incident with $v$ in $G$. Since $G$ is $(n-1)$-critical, $\td(G-e)=n-2$, and by Lemma~\ref{lem: high td forb sub lists}, $G-e$ contains an induced $3K_1$ or $2K_2$ while $G$ contains neither. In $G$ the edge $e$ is either the central edge in an induced $P_4$ or the edge in an induced $K_2+K_1$. Both of these possibilities lead to contradictions, however; if $v$ is a midpoint of an induced $P_4$, then the non-neighbor of $v$ on the path belongs to $V(G)-N_G(v)-\{v\}$ and hence is adjacent to every vertex of $N_G(v)$, including both neighbors of $v$ on the path, a contradiction. If instead $v$ is an endpoint of the edge in an induced $K_2+K_1$, then $v$'s neighbor and non-neighbor are non-adjacent, a contradiction, since vertices in $N_G(v)$ and vertices outside of $N_G(v)$ are adjacent to each other. We conclude that $n$-vertex, $(n-1)$-critical graphs are 1-unique.
\end{proof}

\section{Separating 1-uniqueness and criticality}

As mentioned in the introduction, the property of 1-uniqueness leads a graph to have many properties in common with critical graphs. In particular, although the converse of Conjecture~\ref{conj: false} is false, in~\cite{BarrusSinkovic15} the authors showed that the following is true.

\begin{thm}(\cite{BarrusSinkovic15})
Let $G$ be a connected 1-unique graph with tree-depth $k$.
\begin{enumerate}
\item[\textup{(a)}] If $v$ is any vertex of $G$, then $\td(G-v) < k$.
\item[\textup{(b)}] If $e$ is any edge of $G$, and $G'$ is obtained from $G$ by contracting edge $e$, then $\td(G') < k$.
\item[\textup{(c)}] There exists a set $S$ of edges of $G$ such that $G-S$ is a $k$-critical graph.
\end{enumerate}
\end{thm}

Given that the operations in parts (a) and (b) are two of the operations allowed in obtaining a minor of a graph, with the third operation being edge deletion, as touched on in part (c), it is natural to wonder if some limit can be imposed on the size of the edge set $S$; if so, 1-unique graphs would in one sense be ``close'' to being critical.

Additional developments detailed in~\cite{BarrusSinkovic15}, as well as the positive result in Theorem~\ref{thm: n-1 critical is 1-unique} above, suggest the question posed in Conjecture~\ref{conj: false} of whether every critical graph is 1-unique.

In this section we provide negative answers to both questions. We begin in Section~\ref{subsec: cycle comp} by showing that no constant bound on $|S|$ is possible, in the sense that we show that there exist dense 1-unique graphs with arbitrarily many more edges than certain critical spanning subgraphs. In Section~\ref{subsec: counterexamples} we then exhibit an infinite family demonstrating the existence of non-1-unique, $k$-critical graphs for all $k  \geq 5$.

\subsection{1-uniqueness and critical spanning subgraphs} \label{subsec: cycle comp}

To see that 1-unique graphs exist from which arbitrarily many edges may be deleted without lowering the tree-depth, consider the complements of cycles. 

\begin{thm}
For every integer $n \geq 5$, the graph $\Compn$ has tree-depth $n-1$. Moreover, $\Compn$ is 1-unique.
\end{thm}
\begin{proof}
Observe that for $n \geq 5$, the graph $\Compn$ is not complete and contains no induced $3K_1$ or $2K_2$; from Lemma~\ref{lem: high td forb sub lists}, we conclude that $\td(\Compn)=n-1$.

We now exhibit a 1-unique labeling of $\Compn$. With the vertices of $\Compn$ denoted by $1,\dots,n$ as before, assign vertices 1 and 2 the labels 1 and 2, respectively, and for $3 \leq i \leq n$, assign vertex $i$ the label $i-1$. Vertices 2 and 3 are the only vertices receiving the same label, and it is easy to verify that every path between these two vertices passes through a vertex receiving a higher label. Thus the labeling described is feasible, and since $\Compn$ is vertex-transitive, $\Compn$ may be feasibly labeled so that any desired vertex receives the unique 1 in the labeling.
\end{proof}

Having established that $\Compn$ is 1-unique, we now define another class of graphs. Let $n=4k$, where $k$ is an integer greater than 1. Let $G_n$ denote the graph obtained from $\Compn$, with vertices denoted by $1,\dots,n$ as before, by deleting all edges of the form $\{2j,2j+2k\}$, where $j \in \{1,\dots,k\}$. In words, we form $G_n$ by proceeding along the vertices in order, alternately deleting and leaving alone the edges of $\Compn$ that correspond to pairs of antipodal vertices in $C_n$. The graph $G_3$ is illustrated in Figure~\ref{fig: G3}.
\begin{figure}
\centering
\includegraphics[width=2in]{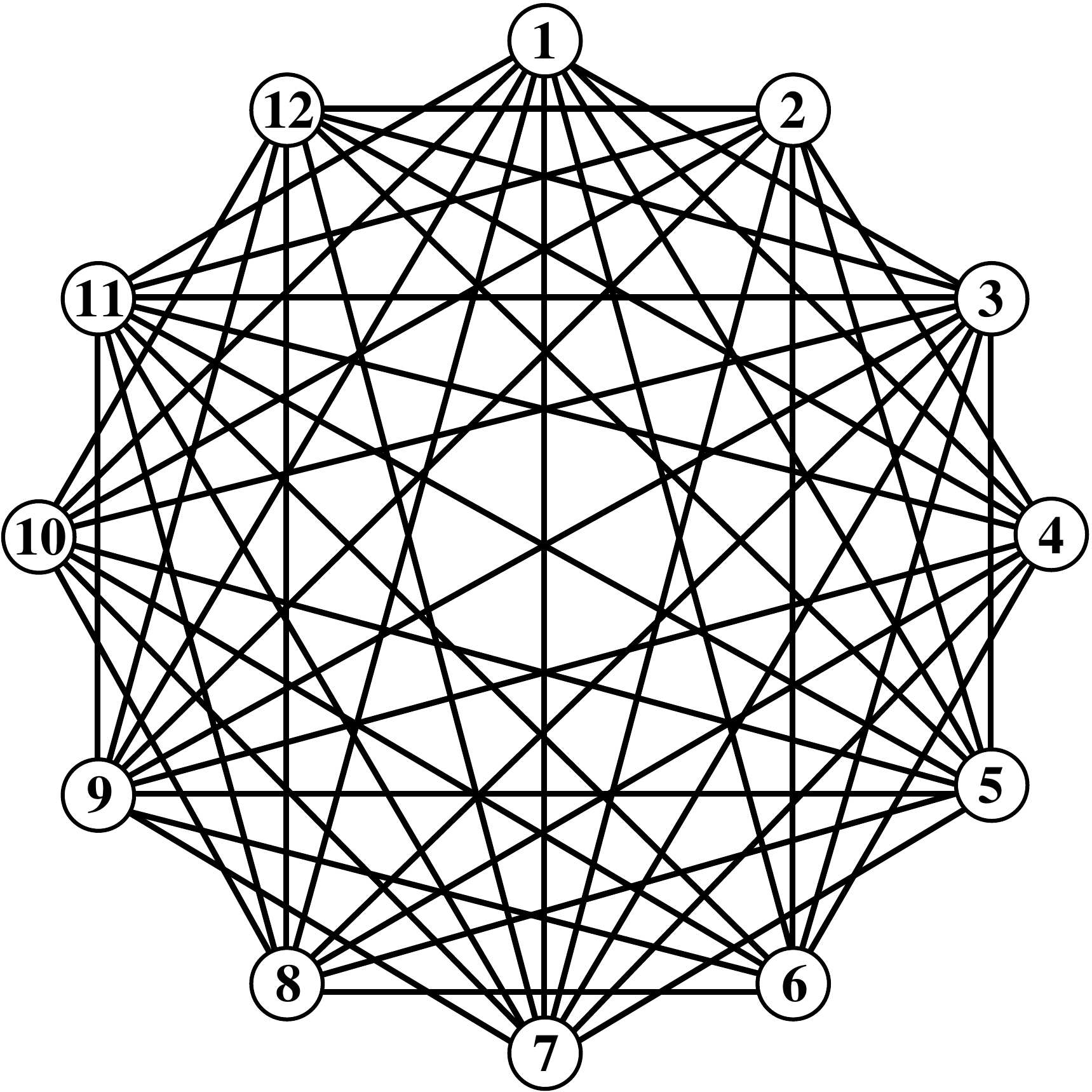}
\caption{The graph $G_3$.}
\label{fig: G3}
\end{figure}

It is straightforward to verify that the complement of $G_n$, which is a cycle in which chords join alternate pairs of antipodal vertices, contains no triangle or 4-cycle; hence $G_n$ is $\{3K_1,2K_2\}$-free, and by Lemma~\ref{lem: high td forb sub lists} we conclude that $\td(G_n) = n-1 = \td(\Compn)$. Noting that $G_n$ has exactly $k$ edges fewer than $\Compn$ yields the following.

\begin{thm} \label{thm: arbitrarily many edges}
For any $k \geq 2$ there exists a 1-unique graph and a spanning subgraph with the same tree-depth but having at least $k$ fewer edges. Hence it is possible to delete arbitrarily many edges from a 1-unique graph before obtaining a critical subgraph.
\end{thm}

\subsection{Non-1-unique critical graphs} \label{subsec: counterexamples}
Having shown in Section~\ref{subsec: cycle comp} that in at least one sense there may be a big difference between 1-unique graphs and critical graphs, we now present a number of counterexamples to Conjecture~\ref{conj: false} that all critical graphs are 1-unique. Portions of these results were presented in~\cite{unpub}.

\begin{figure}
\centering
\includegraphics[width=3.5in]{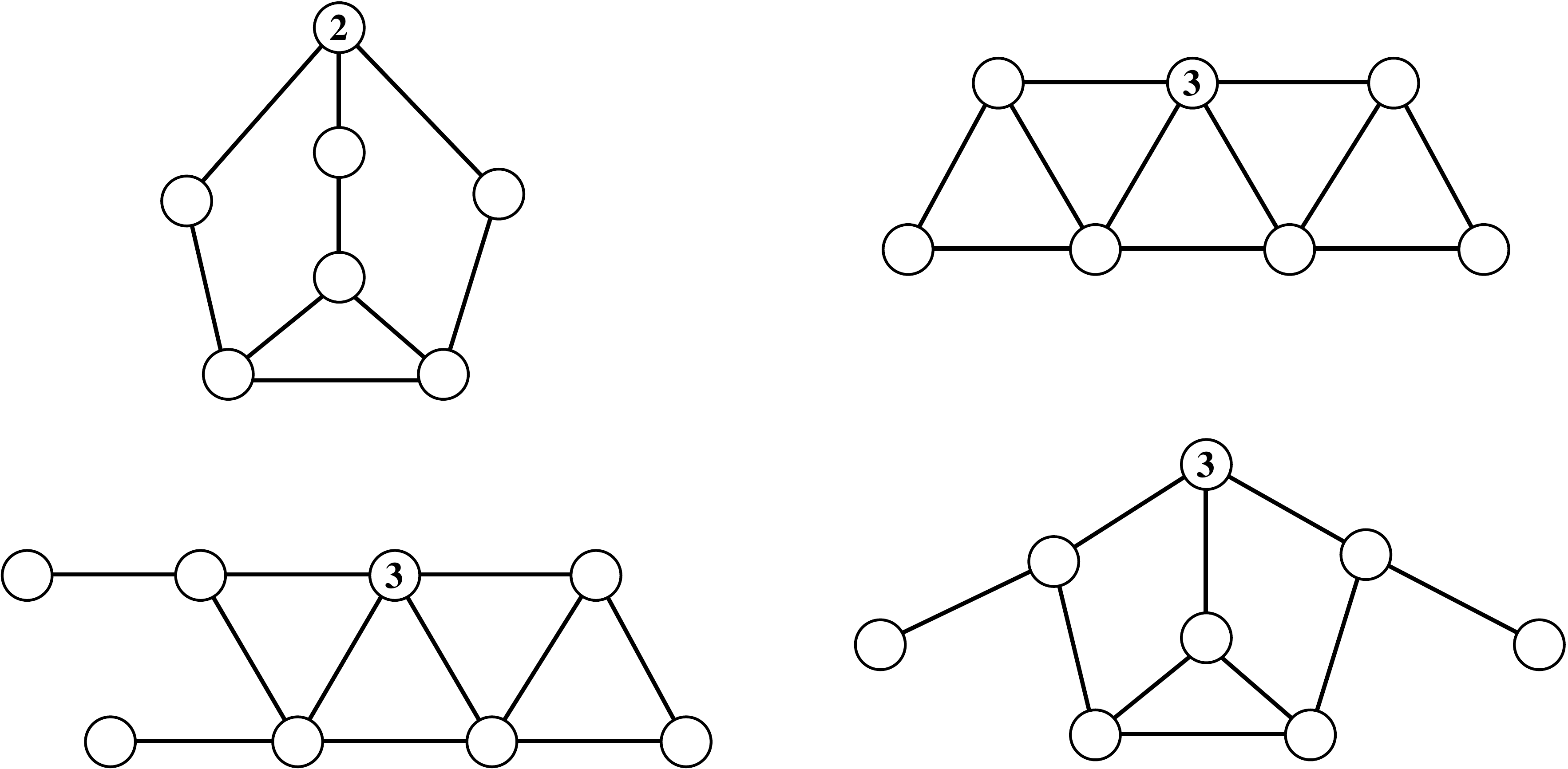}
\caption{Some 5-critical graphs with non-1-unique vertices.}
\label{fig: non-1-unique examples}
\end{figure}
The four graphs in Figure~\ref{fig: non-1-unique examples} are each 5-critical, but in each graph the labeled vertex is not 1-unique. (Instead, the label indicates the smallest label the indicated vertex $v$ can receive in an optimal labeling where $v$ shares its label with no other vertex.) These graphs and others were found using the open source software SageMath, based on an algorithm that uses the ideas of $t$-uniqueness to search for 5-critical graphs, which we now describe briefly.

The algorithm considers a graph $G$ from SageMath's dababase of small graphs. After determining that $G$ has tree-depth 5, for each vertex $v$ in $G$ the algorithm finds the smallest value of $t$ for which $v$ is $t$-unique, if such a value exists; it does this by examining every feasible labeling of $G$ with 5 labels. If each vertex of $G$ has such a value $t$, then the graph is induced-subgraph-critical~\cite{BarrusSinkovic15}. The subgraph-critical graphs are found from the induced-subgraph-critical graphs by determining whether the tree-depth decreases upon deletion of any single edge.  Critical graphs are found among the subgraph-critical graphs by testing edge contractions; our tests are simplified by a result in~\cite{BarrusSinkovic15} that assures us that it suffices to restrict our tests to edges not incident with any $1$-unique vertex.

Note now that if a graph $G$ is subgraph-critical with at most one vertex which is not $1$-unique, then $G$ is critical. Indeed, all 5-critical counterexamples to Conjecture~\ref{conj: false} with 9 or fewer vertices, like the ones in Figure~\ref{fig: non-1-unique examples}, have exactly one vertex which is not $1$-unique.

We can expand the families of counterexamples to include graphs with tree-depths other than 5; in fact, the family we will present contains a non-1-unique $k$-critical graph for any $k \geq 5$; furthermore, the tree-depth of such a graph can differ from the order of the graph by as little as 2, showing that the bound in Theorem~\ref{thm: n-1 critical is 1-unique} cannot be improved.

Before describing the family of counterexamples we need a few simple preliminaries. First, for any positive integer $k$, define a \emph{$k$-net} to be the graph constructed by attaching a single pendant vertex to each vertex of the complete graph $K_k$. The following fact is a special case of Lemma~2.7 in~\cite{unpub}.

\begin{lem}\label{lem:k-net}
A $k$-net has tree-depth $k+1$.
\end{lem}

Next is a result on the tree-depth of the Cartesian product of a complete graph with $K_2$.

\begin{lem}\label{lem: Ka Box K2}
For any positive integer $a$, the graph $K_a \Box K_2$ has tree-depth $\lceil 3a/2\rceil$.
\end{lem}
\begin{proof}
The claim is easily verified for $a \in \{1,2\}$, so suppose $a \geq 3$. Let $V_1$ and $V_2$ denote the disjoint vertex sets of the two induced copies of $K_a$ in $K_a \Box K_2$. We may group the vertices of $K_a \Box K_2$ into pairs $\{u,u'\}$, where $uu'$ is an edge and $u$ and $u'$ are elements of $V_1$ and $V_2$, respectively. Any cutset $T$ in $K_a \Box K_2$ must contain at least one vertex from each such pair, so $|T| \geq a$. Moreover, since $K_a \Box K_2$ has independence number 2, after deleting any cutset the resulting graph has exactly two components, which must be complete subgraphs, the larger of which has at least $a - |T|/2$ vertices. It follows that the tree-depth of $K_a \Box K_2$ is at least $a + |T|/2$, which is at least $\lceil 3a/2\rceil$.

To demonstrate equality, let $T$ be a subset of $V(K_a \Box K_2)$ consisting of $\lfloor a/2 \rfloor$ vertices from $V_1$ and $\lceil a/2 \rceil$ vertices from $V_2$, with no vertex in $T \cap V_1$ adjacent to any vertex in $T \cap V_2$. Label the vertices in $T$ injectively with labels from $\{\lceil a/2 \rceil+1, \dots, \lceil 3a/2\rceil\}$, and in each of $V_1 - T$ and $V_2 - T$, injectively label the vertices with $\{1,\dots,\lceil a/2 \rceil\}$. It is straightforward to verify that this is a feasible labeling using the appropriate number of colors.
\end{proof}

\begin{thm}
For any $n \geq 4$, let $H_n$ be the graph obtained by subdividing (once) all edges incident with a single vertex $v$ of $K_{n}$. The graph $H_n$ is $(n+1)$-critical but not $1$-unique; moreover, $v$ is the only non-1-unique vertex in $H_n$.
\end{thm}
\begin{proof}
In the following, let $A_n$ denote the vertices of degree $2$ incident with $v$ in $H_n$, and let $B_n$ denote $V(H_n)-v-A_n$; note that the $n-1$ vertices in $B_n$ form a clique in $H_n$, and each vertex in $B_n$ is adjacent exactly to the other vertices of $B_n$ and to a single vertex in $A_n$ (with each vertex in $A_n$ having a single neighbor in $B_n$).

To see that $\td(H_n) \leq n+1$, injectively label the vertices of $B_n$ with labels $2,\dots,n$, label each vertex in $A_n$ with $1$, and label $v$ with $n+1$. Under this labeling only vertices in $A_n$ receive a common label, and each path joining two vertices in $A_n$ contains a vertex outside $A_n$, which has a higher label than $1$.

For convenience in proving that $\td(H_n) \geq n+1$, we now construct a graph $H_3$ in the same way that $H_n$ is defined for $n \geq 4$; note that $H_3$ is isomorphic to $C_5$. By induction we show that $\td(H_n) \geq n+1$ for all $n \geq 3$). 

Observe that $H_3$ has tree-depth $4$, as desired. Now suppose that for some integer $k \geq 3$ we have $\td(H_k) \geq k+1$. Now consider the result of deleting a vertex from $H_{k+1}$. If the vertex deleted is $v$, the remaining graph is isomorphic to a $k$-net, which by Lemma~\ref{lem:k-net} has tree-depth $k+1$. Deleting any vertex from $A_{k+1}$ or from $B_{k+1}$, along with its neighbor in the other set, leaves a copy of $H_k$, which by our induction hypothesis has tree-depth $k+1$. Thus $\td(H_{k+1}) \geq 1 + \td(H_k) \geq (k+1)+1$, as desired.

We now show that $H_n$ is critical. Note that if $u$ is any vertex in $A_n$, and if $w$ is the neighbor of $u$ in $B_n$, then each of $H_n - uv$ and $H_n-uw$ may be feasibly colored by labeling $v$ and $w$ with $2$, labeling all of $A_n$ with $1$, and injectively labeling the vertices of $B_n - u$ with colors from $3,\dots,n$.

If $w,w'$ are vertices in $B_n$, we feasibly color $H_n - ww'$ by labelling $w$ and $w'$ with $1$, labeling all vertices in $A_n$ with $2$, labeling $v$ with $3$, and injectively labeling the vertices of $B_n - \{w,w'\}$ with colors from $4,\dots,n$.

Contracting an edge of $H_n$ that is incident with a vertex in $A_n$ yields a graph isomorphic to that obtained by adding to $H_{n-1}$ a vertex $w'$ adjacent to the analogous vertex $v$ and to all vertices of $B_{n-1}$; we feasibly color this graph by labeling $w'$ and all vertices in $A_{n-1}$ with $1$, labeling $v$ with $2$, and injectively labeling vertices in $B_{n-1}$ with $\{3,\dots,n\}$.

Contracting an edge $w_1w_2$ of $H_n$, where $w_1,w_2 \in B_n$, yields a graph that can be feasibly colored in the following way: label all vertices of $A_n$ with 1, label the vertex replacing $w_1$ and $w_2$ with 2, label $v$ with 3, and injectively label the vertices of $B_n-\{w_1,w_2\}$ with $\{4,\dots,n\}$. Having shown now that deleting or contracting any edge results in a graph with smaller tree-depth (and, it follows, the same holds if we delete any vertex), we conclude that $G$ is $(n+1)$-critical.

Now by Theorem~\ref{thm:starclique}, $v$ will be 1-unique if and only if performing a star-clique transform at  $v$ in $H_n$ yields a graph with a lower tree-depth. Observe that a star-clique transform on $v$ actually yields $K_{n-1} \Box K_2$. By Lemma~\ref{lem: Ka Box K2}, $\td(K_{n-1}\Box K_2) = \lceil \frac{3}{2}(n-1)\rceil$, which is at least $n+1$ for $n \geq 4$, rendering $v$ non-1-unique. Note, however, that a star-clique transform on any vertex of $A_n$ or $B_n$ has the same effect as contracting an edge between $A_n$ and $B_n$ in $H_n$, which lowers the tree-depth, as we verified above; hence all vertices of $H_n$ other than $v$ are 1-unique.
\end{proof}

\section{A family of dense 1-unique critical graphs}

In this section we present a pleasing family of graphs that have appeared in the literature but were previously not known to be critical with respect to tree-depth. Though the previous two sections have established results separating criticality from 1-uniqueness, our proof in this section will use 1-uniqueness to efficiently establish criticality.

For any positive integer $k$, the \emph{Andr\'{a}sfai graph} $\And(k)$ is defined to be the graph with vertex set $V=\{0,\dots,3k-2\}$ where edges are defined to be pairs $i,j$ (assume that $i>j$) such that $i-j$ is congruent to 1 modulo 3. The graph $\And(5)$ is shown in Figure~\ref{fig: And(5)}. Andr\'{a}sfai graphs are discussed in~\cite{Andrasfai64,GodsilRoyle}. It is easy to see that $\And(k)$ is a circulant graph and a Cayley graph. As we will see, these graphs also have pleasing properties regarding tree-depth.
\begin{figure}
\centering
\includegraphics[width=2in]{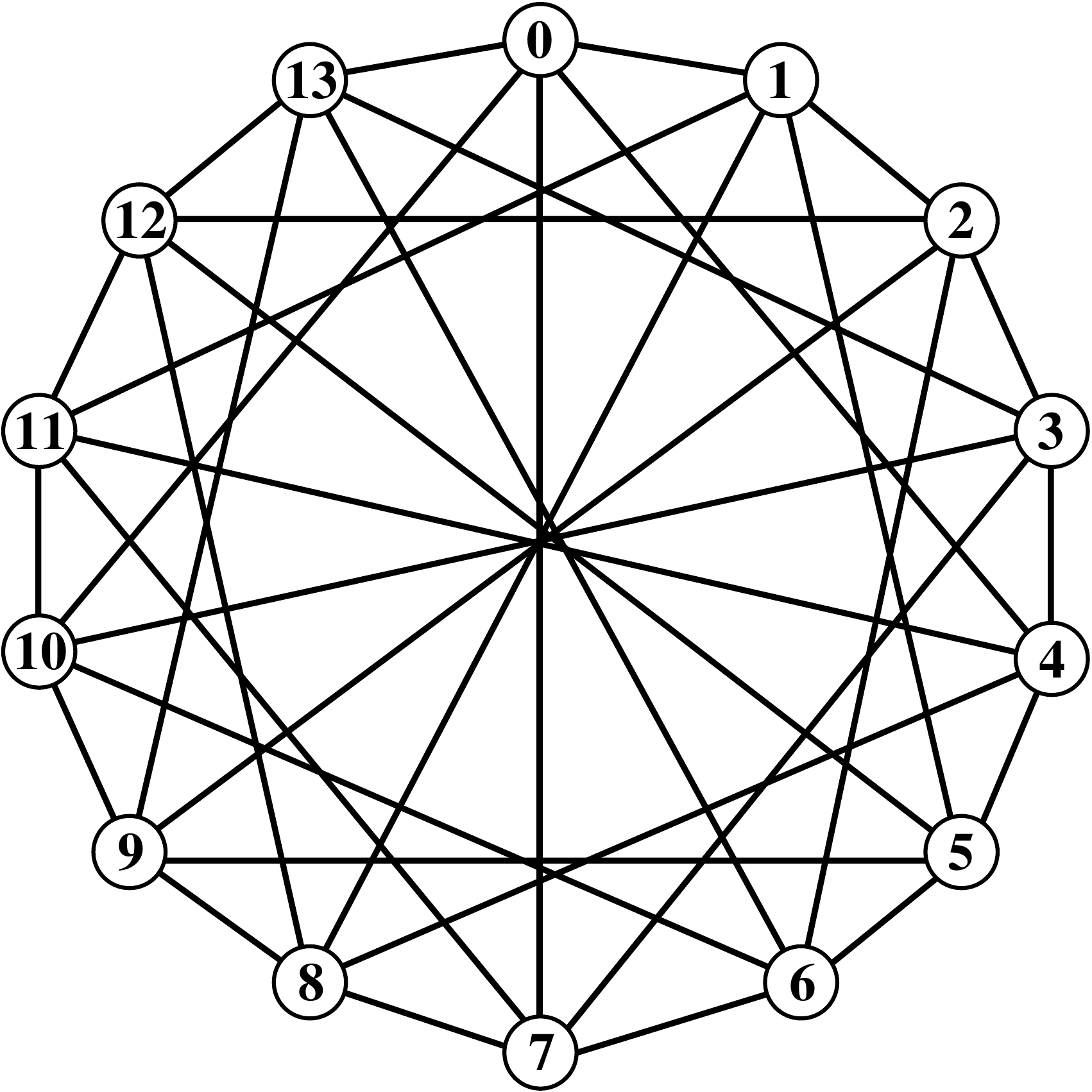}
\caption{The Andr\'{a}sfai graph $\And(5)$.}
\label{fig: And(5)}
\end{figure}

In the following, let $\And(k)-v$ denote the graph obtained by deleting a vertex from $\And(k)$; since $\And(k)$ is vertex-transitive, this graph is well-defined up to isomorphism.

\begin{lem}\label{lem: Andrasfai connectivity}
For all $k \geq 1$, the graph $\And(k)$ is $k$-connected.
\end{lem}
\begin{proof}
By Menger's Theorem it suffices to show that between any two vertices in $\And(k)$ there are at least $k$ pairwise internally disjoint paths. We prove this by induction on $k$. When $k=1$, the graph $\And(1) = K_2$, and there is clearly a path between the two vertices. Suppose now that for some integer $j\geq 1$, the graph $\And(j)$ has $j$ pairwise internally disjoint paths between any two distinct vertices. We now consider the number of internally disjoint paths between an arbitrary pair of vertices in $\And(j+1)$. Since this graph is vertex-transitive, without loss of generality we may assume that 0 is one of the vertices of the pair; if $a$ denotes the other vertex, then by symmetry we may assume that $1 \leq a \leq 3j-2$. Observe that the induced subgraph on vertices $\{0,\dots,3j-2\}$ is isomorphic to $\And(j)$, so by the induction hypothesis there exists a set of at least $j$ pairwise internally disjoint vertices joining $1$ and $a$ and using only vertices from $\{0,\dots,3j-2\}$. Now note that vertices $0$ and $a$ each have a neighbor in $\{3j-1,3j,3j+1\}$, so we may find a path from $0$ to $a$ whose internal vertices are drawn from this set; this path is necessarily internally disjoint from each of the earlier $j$ paths. Thus $\And(j+1)$ contains at least $j+1$ pairwise internally disjoint paths between any two vertices; by induction our proof is complete.
\end{proof}

We now define a useful labeling of the vertices of $\And(k)$.

\begin{defn}
The \emph{standard labeling} of $\And(k)$ is a function $r: V(\And(k)) \to \{1,2,\dots,2k\}$ given as follows:

\[r(x) = \begin{cases}
1 & \textup{if } x = 0;\\
2 & \textup{if }x>0 \text{ and }x \equiv 0 \!\!\!\!\pmod 3;\\
\frac{2x+4}{3} & \textup{if }x \equiv 1 \!\!\!\!\pmod 3;\\
\frac{2x+5}{3} & \textup{if }x \equiv 2 \!\!\!\!\pmod 3.
\end{cases}\]

In words, the standard labeling assigns label 1 to vertex 0, labels vertices $1,2,4,5,7,8,\dots,3k-2$ (i.e. all vertices that are not multiples of 3) in order, injectively, with the labels $2,3,\dots,2k$, and assigns label 2 to all other vertices.
\end{defn}
\begin{figure}
\centering
\includegraphics[width=2in]{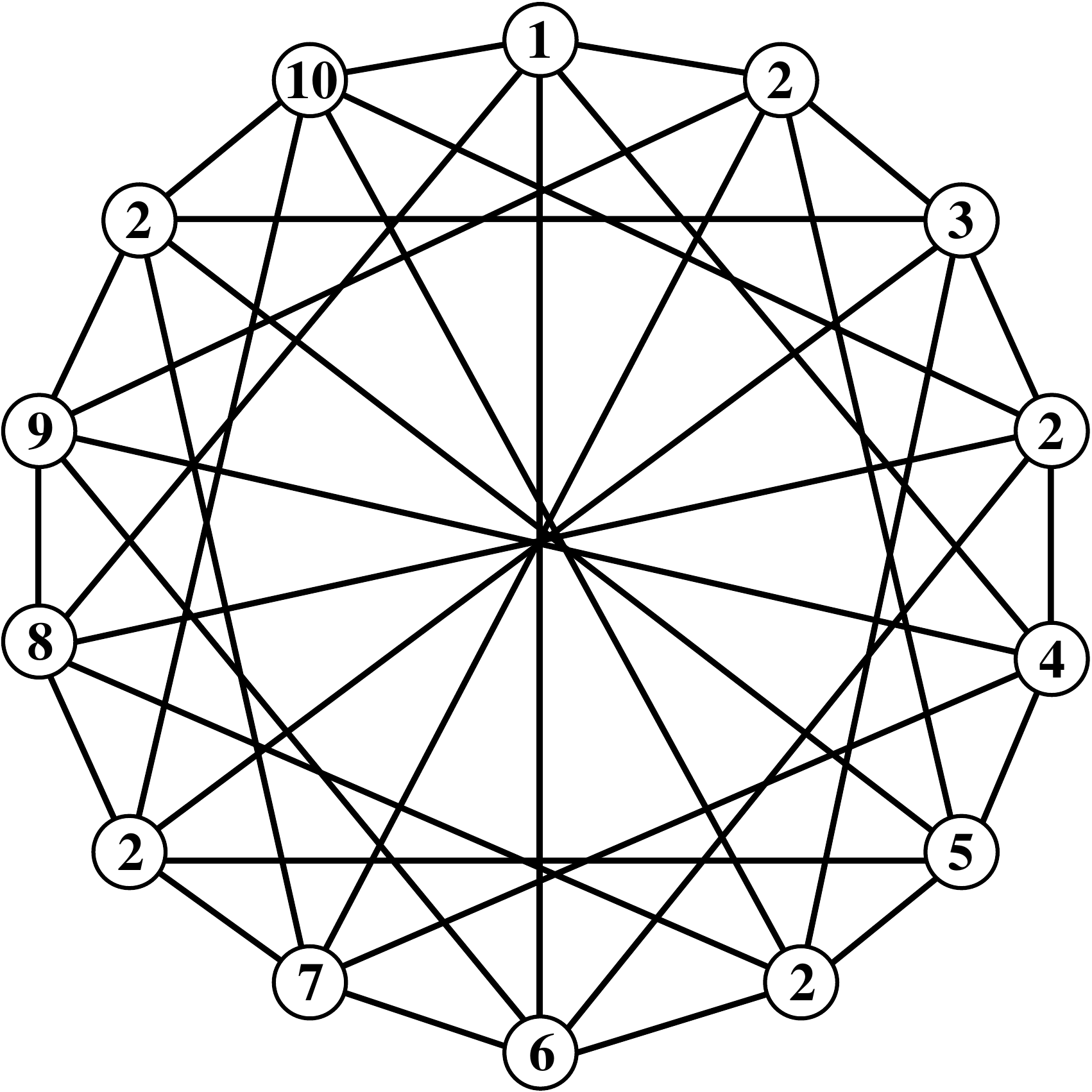}
\caption{The standard labeling of $\And(5)$.}
\label{fig: And(5) labeled}
\end{figure}

Figure~\ref{fig: And(5) labeled} shows a standard labeling of $\And(5)$. Observe that in the standard labeling of $\And(k)$ the only repeated label is 2, no two vertices with label 2 are adjacent (since the difference of any two vertices labeled 2 is a multiple of 3 or is 1 less than a multiple of 3). Moreover, the only vertex adjacent to a vertex with label 1 is vertex 1, and any path beginning at vertex 1 and passing through vertex 0 must immediately afterward pass through another vertex with a higher label (since the last vertex has a number that is congruent to 1 modulo 3). It follows that any path between two vertices with the same label must include a vertex having a higher label than that of the endpoints, so the standard labeling fits the conditions of a labeling of the vertices, though we have not yet shown that it is an optimal labeling.

\begin{thm} \label{thm: td Andrasfai}
For all $k \geq 1$, \[\td(\And(k)) = 2k \quad {\text and } \quad \td(\And(k)-v) = 2k-1.\]
\end{thm}
\begin{proof}
Note that if $\td(\And(k)) = 2k$, it easily follows that $\td(\And(k)-v) = 2k-1$, so we restrict our attention to the graphs $\And(k)$.

Since the standard labeling of $\And(k)$ has $2k$ as its highest label value, $\td(\And(k)) \leq 2k$. To show that $\td(\And(k)) \geq 2k$ we proceed by induction. Note that $\And(1)$ has tree-depth $2$, and suppose that $\td(\And(j))=2j$ for some positive integer $j$.

Now fix an optimal labeling of $\And(j+1)$. By Lemma~\ref{lem: Andrasfai connectivity}, $\And(j+1)$ is $(j+1)$-connected and hence the highest $j+1$ labels appear only once in the labeling. By an application of the pigeonhole principle, there must exist three consecutive vertices in $\{0,\dots,3j+1\}$ (where consecutivity is determined modulo $3j+2$) such that the labels on these vertices include two values from the highest $j+1$ labels. Exchange the labels on these vertices with those on the vertices receiving the highest two labels; since the only repeated labels in the labeling have a value not from the $j+1$ highest values, the resulting labeling is still a feasible, optimal labeling of $\And(j+1)$. Now by symmetry, we may assume that the highest two labels occur among the vertices $\{3j-1, 3j, 3j+1\}$. Since the induced subgraph on vertices $\{0,\dots,3j-2\}$ is equal to $\And(j)$, by the induction hypothesis the labeling must use at least $2j$ labels on this subgraph, which means that the optimal labeling of $\And(j+1)$ uses at least $2j+2$ distinct labels, and the induction is complete.
\end{proof}

\begin{thm} \label{thm: And is 1-unique}
For all $k \geq 1$, both $\And(k)$ and $\And(k)-v$ are 1-unique.
\end{thm}
\begin{proof}
We observe that the standard labeling of $\And(k)$ places the label 1 only on the vertex 0; since $\And(k)$ is vertex-transitive, it follows that $\And(k)$ is 1-unique. To show that $\td(\And(k)-v)$ is 1-unique, we may assume that $k \geq 2$, since the claim is clearly true when $k=1$. It suffices by symmetry to show that for all $v \in \{1,\dots,3k-2\}$ there is a labeling of $\And(k)-v$ using $2k-1$ distinct labels and placing a unique label of 1 on vertex 0. We prove this in cases.

\textit{Case: $v$ is not 1 and not a multiple of 3.} In this case we begin by labeling $\And(k)$ with the standard labeling. We then delete vertex $v$ and reduce all labels higher than that of $v$ by 1. The result is a 1-unique labeling of $\And(k)-v$ using $2k-1$ labels, as desired.

\textit{Case: $v$ is 1 or a multiple of 3.} In this case $3k-1-v$ is neither 1 nor a multiple of 3. Since there is an automorphism $\phi:V(\And(k)) \to V(\And(k))$ mapping each vertex $x$ to $3k-1-x$ (modulo $3k-1$), we simply label each vertex $x$ of $\And(k)-v$ with $r(\phi(x))$, where $r$ is the standard labeling of $\And(k)$, producing a labeling with the desired properties.
\end{proof}

\begin{thm}
For all $k \geq 1$, both $\And(k)$ and $\And(k)-v$ are critical.
\end{thm}
\begin{proof}
Since both $\And(k)$ and $\And(k)-v$ are 1-unique, it suffices by Theorem~\ref{thm: And is 1-unique} to show that deleting any edge from $\And(k)$ or from $\And(k)-v$ lowers the tree-depth.

We consider $\And(k)$ first. Suppose that the deleted edge is $uw$. Let $x$ be a vertex adjacent to neither $u$ nor $w$ in $\And(k)-uw$ (since $u$ and $w$ are each adjacent to exactly one of every three consecutive vertices, such a vertex exists). Label all neighbors of $x$ with 1, label $w$ and $u$ both with 2, and label the remaining $2k-3$ vertices injectively with labels from $\{3,\dots,2k-1\}$.

We claim that this labeling is feasible. Note that since $\And(k)$ contains no triangles, no vertices with label 1 are adjacent, and $\And(k)-uw$ contains no path of length 2 joining $u$ and $w$. It follows that any path between vertices with the same label must be longer than 1 edge if the label is 1, and longer than 2 edges if the label is 2. Such a path must then contain a higher label on an interior vertex than the label on the endpoints is. Thus the labeling is feasible, and $\td(\And(k) - uw) \leq 2k-1 < \td(\And(k))$.

We now show that $\And(k)-v$ is critical. This is clearly verified for $k \leq 2$, so assume $k \geq 3$. As before, it suffices to show that deleting an arbitrary edge lowers the tree-depth. Equivalently, we show that for an arbitrary edge $e$ and vertex $v$ of $\And(k)$, where $e$ is not incident with $v$, the tree-depth of $\And(k) - e - v$ is at most $2k-2$. As before, let $x$ be a vertex other than $v$ that is adjacent to neither endpoint of $e$ (such a vertex exists because $k \geq 3$). We label the neighbors of $x$ with $1$, the former endpoints of $e$ with $2$, and each of the remaining $2k-4$ vertices with a distinct label from $\{3,\dots,2k-2\}$. The same arguments used above for $\And(k)-uw$ are valid in showing that this labeling is feasible; hence $\td(
\And(k)-e-v) \leq 2k-2 < \td(\And(k)-v)$, and our proof is complete.
\end{proof}

In conclusion we remark that the Andr\'{a}sfai graphs join the cycles of order $2^n+1$ and complete graphs as critical graphs having the remarkable property that deleting any vertex yields another critical graph, something that is not true of  critical graphs in general. Interestingly, in each of the graphs in Figure~\ref{fig: non-1-unique examples} (the non-1-unique critical graphs), deleting the non-1-unique vertex from the critical graph yields another critical graph; this seems to hold for several similar counterexamples to Conjecture 1.2 that have a single non-1-unique vertex.

The Andr\'{a}sfai graphs, cycles of order $2^n+1$, and complete graphs, along with the graphs $\overline{C_6}$ and $\overline{C_7}$ are all 1-unique, critical circulant graphs (though cycle-complements in general are not critical, as shown in Section~\ref{subsec: cycle comp}). At present these graphs are the only critical graphs the authors are aware of that are circulant, vertex-transitive, or even regular, though there are almost surely other classes of examples. Though we have seen that  Conjecture~\ref{conj: false} is false for graphs in general, it would be interesting to know whether it holds for circulant graphs (or vertex-transitive or regular graphs).

With the failure of Conjecture~\ref{conj: false}, it also remains to find a different approach to prove the maximum degree property appearing in Conjecture~\ref{conj: order, max degree}  for graphs in general.

\section*{Acknowledgments}
The authors thank A.~Giannopoulou for mentioning the critical example of the triangular prism, which is $\overline{C_6}$, in a personal communication.

\bibliography{treedepth}
\bibliographystyle{elsarticle-num}

\end{document}